\documentclass[11pt]{article}
\usepackage[a4paper,margin=26mm]{geometry}
\usepackage{amsmath,amssymb,amsthm,mathtools}
\usepackage{microtype}
\usepackage{tikz}
\newcommand{\drawpartitiongrid}[2]{%
  \draw (0,0) -- (#2,0);
  \foreach \row/\len in {#1}{%
    \draw (0,{\row+1}) -- (\len,{\row+1});
    \foreach \col in {0,...,\len}{\draw (\col,\row) -- ++(0,1);}%
  }%
}
\usepackage[hidelinks]{hyperref}
\hypersetup{
  pdftitle={Macdonald principal specialization ratios and their generating polynomials},
  pdfauthor={Tatsushi Shimazaki}
}

\numberwithin{equation}{section}
\newtheorem{theorem}{Theorem}[section]
\newtheorem{proposition}[theorem]{Proposition}
\newtheorem{lemma}[theorem]{Lemma}
\newtheorem{corollary}[theorem]{Corollary}
\theoremstyle{definition}
\newtheorem{definition}[theorem]{Definition}
\theoremstyle{plain}

\begin{document}
\pagestyle{plain}
\flushbottom

\begin{center}
{\Large\bfseries Macdonald principal specialization ratios\\ and their generating polynomials}\par
\vspace{1mm}
{\large Tatsushi Shimazaki\footnotemark}\par
\end{center}
\footnotetext{Liberal Arts, National Institute of Technology, Akashi College, Akashi, Hyogo 674--8501, Japan.\\
E-mail: \href{mailto:t.shimazaki@akashi.ac.jp}{t.shimazaki@akashi.ac.jp}.
}
\vspace{3mm}

\begin{abstract}
We establish independence of $q$ as a necessary condition for coefficientwise nonnegativity of normalized principal specialization ratios of monic Macdonald polynomials. For arbitrary pairs of partitions, we characterize independence of $q$. For denominator partitions with largest and smallest parts differing by at least two, every ratio independent of $q$ is a power of $t$. We classify the nonnegative ratios under addition of maximal rows and columns of fixed height. For rows of length at least two, the classification is equivalent to rectangular complementation. For arbitrary initial partitions, the generating polynomials of successive row ratios at the smallest number of variables admissible for all stages have simple negative real zeros with lower bounds for the ratios of consecutive absolute zeros. We derive a reciprocity criterion, bounds for gamma coefficients, and the corresponding Jack limit. For generalized staircases, we determine shortest horizontal and conjugate vertical strip sequences and obtain terminating basic hypergeometric representations.
\end{abstract}

\section{Introduction}

Macdonald's principal specialization formula expresses evaluations of
monic Macdonald polynomials as products in $q$ and
$t$~\cite{Macdonald,Noumi}. Cancellations among the factors in their ratios do not decide
coefficientwise nonnegativity. We prove independence of $q$ for every
normalized ratio with nonnegative coefficients, using the limit
$t\to1^-$ for fixed $q$. We characterize $q$ independence by counting
differences between parts at fixed row distances. For denominator
partitions with largest and smallest parts differing by at least two,
every ratio independent of $q$ is a power of $t$. For smaller differences,
we reduce the remaining nonnegativity question to quotients of
Gaussian polynomials~\cite{GatzweilerKrattenthaler}.

We classify the nonnegative ratios obtained by adjoining maximal rows
or columns of fixed height. For rows of length at least two, the
classification is equivalent to rectangular complementation. For
columns, nonnegativity requires the minimum number of variables,
except in the elementary symmetric case. We determine the unique
shortest sequence of horizontal strips joining consecutive generalized
staircases and its conjugate vertical sequence. At the minimum number
of variables, the horizontal ratios yield a terminating basic
hypergeometric polynomial.

For an arbitrary initial partition, the generating polynomial of
successive row ratios at the smallest number of variables admissible for
all stages has negative real zeros~\cite{MPM}. We give a direct proof of
their simplicity and establish lower bounds for the ratios of consecutive
absolute zeros using Schur--Szeg\H{o} composition and finite multiplier
sequences. We characterize reciprocity of
the centered polynomial through rectangular complementation. For
generalized staircases, reciprocal pairs of zeros yield bounds on the
coefficients of its symmetric expansion. The Jack limit leads to
Jacobi polynomials.

The organization of this paper is as follows. Section~2 gives
preliminaries. Section~3 establishes the specialization ratios, and
Section~4 treats their generating polynomials.

\section{Preliminaries}

\subsection{Partitions and Young diagrams}

A partition $\lambda=(\lambda_1,\lambda_2,\ldots)$ is a weakly
decreasing sequence of nonnegative integers with finitely many nonzero
parts. Its Young diagram in English notation is
\[
\lambda=\{(i,j)\in\mathbb Z_{>0}^2:1\le j\le\lambda_i\}.
\]
Rows are indexed from top to bottom and columns from left to right. The size and length of $\lambda$ are
\[
\lvert\lambda\rvert=\sum_{i\ge1}\lambda_i,
\qquad
\ell(\lambda)=\lvert\{i:\lambda_i>0\}\rvert.
\]
The empty partition is denoted by $\varnothing$. The conjugate partition $\lambda'$ is defined by
\[
\lambda'_j=\lvert\{i:\lambda_i\ge j\}\rvert.
\]

For a box $(i,j)\in\lambda$, its arm, coarm, leg, and coleg lengths are
\[
\begin{aligned}
a_\lambda(i,j)&=\lambda_i-j, & a'_\lambda(i,j)&=j-1,\\
l_\lambda(i,j)&=\lambda'_j-i, & l'_\lambda(i,j)&=i-1.
\end{aligned}
\]
Define
\[
n(\lambda)=\sum_{i\ge1}(i-1)\lambda_i.
\]
For a partition with at most $n$ parts, append zeros to obtain $n$ parts.
The multiplicity of $j\ge0$ in this sequence is $m_j(\lambda)$.
The multiplicity $m_0(\lambda)=n-\ell(\lambda)$ depends on $n$.

For partitions $\mu$ and $\lambda$, write $\mu\subseteq\lambda$ for $\mu_i\le\lambda_i$ for all $i$. For $\mu\subseteq\lambda$, the skew diagram $\lambda/\mu$ is the set difference of their Young diagrams. A skew diagram is a horizontal strip provided it contains at most one box in each column. It is a vertical strip provided it contains at most one box in each row.

For $m,n\in\mathbb Z_{>0}$, the notation $(m^n)$ denotes the rectangular partition with $n$ parts equal to $m$. For $\lambda\subseteq(m^n)$, write $\lambda=(\lambda_1,\ldots,\lambda_n)$. The rectangular complement of $\lambda$ in $(m^n)$ is
\[
(m-\lambda_n,m-\lambda_{n-1},\ldots,m-\lambda_1).
\]
Geometrically, the rectangular complement is obtained by rotating the skew diagram $(m^n)/\lambda$ by $180^\circ$.

For $\lambda=(6,4,4,4,2,2,2)\subseteq(6^9)$, the rectangular complement is
$(6,6,4,4,4,2,2,2)$, as illustrated below. The shaded cells in the left diagram form $(6^9)/\lambda$.

\begin{center}
\begin{tikzpicture}[x=0.340cm,y=-0.340cm, line width=0.35pt]
\definecolor{compfill}{RGB}{250,232,204}
\definecolor{compline}{RGB}{188,126,54}
\begin{scope}[shift={(0,0)}]
\foreach \row/\len in {0/6,1/4,2/4,3/4,4/2,5/2,6/2,7/0,8/0}{
  \foreach \col in {0,...,5}{
    \ifnum\col<\len\else\fill[fill=compfill] (\col,\row) rectangle +(1,1);\fi
  }}
\drawpartitiongrid{0/6,1/6,2/6,3/6,4/6,5/6,6/6,7/6,8/6}{6}
\node[align=center] at (3,10.4) {$\lambda\subseteq(6^9)$};
\end{scope}
\begin{scope}[shift={(13.5,0)}]
\draw[dashed,gray] (0,0) -- (6,0) -- (6,1);
\draw[dashed,gray] (0,0) -- (0,7);
\foreach \row/\len in {0/6,1/4,2/4,3/4,4/2,5/2,6/2,7/0,8/0}{
  \foreach \col in {0,...,5}{
    \ifnum\col<\len\else
      \fill[fill=compfill] (\col,\row) rectangle +(1,1);
    \fi
  }}
\foreach \row/\start in {1/4,2/4,3/4,4/2,5/2,6/2,7/0,8/0}{
  \draw (\start,{\row+1}) -- (6,{\row+1});
  \foreach \col in {\start,...,6}{\draw (\col,\row) -- ++(0,1);}
}
\draw (4,1) -- (6,1);
\draw (2,4) -- (4,4);
\draw (0,7) -- (2,7);
\end{scope}
\begin{scope}[shift={(27.0,0.5)}]
\foreach \row/\len in {0/6,1/6,2/4,3/4,4/4,5/2,6/2,7/2}{
  \foreach \col in {0,...,5}{
    \ifnum\col<\len
      \fill[fill=compfill] (\col,\row) rectangle +(1,1);
    \fi
  }}
\drawpartitiongrid{0/6,1/6,2/4,3/4,4/4,5/2,6/2,7/2}{6}
\node at (3,9.7) {\small rectangular complement of $\lambda$};
\end{scope}
\draw[->,thick,compline] (7.1,4.5) -- (12.3,4.5)
  node[midway,above=3pt,text=compline] {\scriptsize complement};
\draw[->,thick,compline] (20.6,4.5) -- (25.8,4.5)
  node[midway,above=3pt,text=compline] {\scriptsize $180^\circ$};
\end{tikzpicture}
\end{center}

\subsection{Generalized staircase partitions}

\begin{definition}
Let $k,w,s\in\mathbb Z_{>0}$ and $0\le g\le s$. Define
\begin{align*}
\delta_{k,g}^{(w,s)}
=\bigl((wk)^g,(w(k-1))^s,(w(k-2))^s,\ldots,w^s\bigr).
\end{align*}
For $k=1$, the terms following $(wk)^g$ are absent. The part
$(wk)^0$ is omitted. We call $\delta_{k,g}^{(w,s)}$ a generalized
staircase partition. For $g=0$, omit the index $g$ and write
\[
\delta_k^{(w,s)}=\delta_{k,0}^{(w,s)}.
\]
\end{definition}

For $w=s=1$, write $\delta_k=\delta_k^{(1,1)}$. Replacing each box of $\delta_k$ with a rectangle of shape
$(w^s)$ gives the Young diagram of $\delta_k^{(w,s)}$. The horizontal steps have
width $w$, and the vertical steps have height $s$. The ordinary
staircase $\delta_3=(2,1)$ and the generalized staircases
$\delta_{3,g}^{(2,3)}$ for $0\le g\le3$ are shown below.

\begin{center}
\begin{tikzpicture}[x=0.340cm,y=-0.340cm, line width=0.35pt]
\definecolor{wfill}{RGB}{222,234,248}
\definecolor{sfill}{RGB}{224,243,218}
\begin{scope}[shift={(0,0)}]
\drawpartitiongrid{0/2,1/1}{2}
\end{scope}
\begin{scope}[shift={(6,0)}]
\foreach \row/\col in {0/2,0/3,1/2,1/3,2/2,2/3,3/1,4/1,5/1}{\fill[wfill] (\col,\row) rectangle +(1,1);}
\foreach \row/\col in {1/1,2/0,2/1,3/0,4/0,5/0}{\fill[sfill] (\col,\row) rectangle +(1,1);}
\drawpartitiongrid{0/4,1/4,2/4,3/2,4/2,5/2}{4}
\end{scope}
\begin{scope}[shift={(15,0)}]
\foreach \col in {4,5}{\fill[fill=violet!22] (\col,0) rectangle +(1,1);}
\foreach \col in {2,3}{\fill[fill=violet!22] (\col,3) rectangle +(1,1);}
\foreach \col in {0,1}{\fill[fill=violet!22] (\col,6) rectangle +(1,1);}
\drawpartitiongrid{0/6,1/4,2/4,3/4,4/2,5/2,6/2}{6}
\end{scope}
\begin{scope}[shift={(26,0)}]
\foreach \col in {4,5}{\fill[fill=violet!22] (\col,1) rectangle +(1,1);}
\foreach \col in {2,3}{\fill[fill=violet!22] (\col,4) rectangle +(1,1);}
\foreach \col in {0,1}{\fill[fill=violet!22] (\col,7) rectangle +(1,1);}
\drawpartitiongrid{0/6,1/6,2/4,3/4,4/4,5/2,6/2,7/2}{6}
\end{scope}
\begin{scope}[shift={(37,0)}]
\foreach \col in {4,5}{\fill[fill=violet!22] (\col,2) rectangle +(1,1);}
\foreach \col in {2,3}{\fill[fill=violet!22] (\col,5) rectangle +(1,1);}
\foreach \col in {0,1}{\fill[fill=violet!22] (\col,8) rectangle +(1,1);}
\drawpartitiongrid{0/6,1/6,2/6,3/4,4/4,5/4,6/2,7/2,8/2}{6}
\end{scope}
\draw[->,black,semithick] (3.0,3.3) -- (5.2,3.3);
\definecolor{gline}{RGB}{126,95,188}
\draw[->,thick,gline] (10.75,3.3) -- (14.25,3.3);
\draw[->,thick,gline] (21.75,3.3) -- (25.25,3.3);
\draw[->,thick,gline] (32.75,3.3) -- (36.25,3.3);
\node at (1,10.7) {$\delta_3$};
\node at (8,10.7) {$\delta_3^{(2,3)}$};
\node at (18,10.7) {$\delta_{3,1}^{(2,3)}$};
\node at (29,10.7) {$\delta_{3,2}^{(2,3)}$};
\node at (40,10.7) {$\delta_{3,3}^{(2,3)}$};
\end{tikzpicture}
\end{center}
Green and blue mark boxes added by vertical and horizontal enlargement
from $\delta_3$, respectively. The black arrow denotes the enlargement,
and the violet arrows mark successive horizontal strips.

The corresponding conjugate diagrams for $g=1,2$ are shown below.
\begin{center}
\begin{tikzpicture}[x=0.340cm,y=-0.340cm, line width=0.35pt]
\begin{scope}[shift={(0,0)}]
\drawpartitiongrid{0/7,1/7,2/4,3/4,4/1,5/1}{7}
\node at (3.5,7.4) {$\bigl(\delta_{3,1}^{(2,3)}\bigr)'$};
\end{scope}
\begin{scope}[shift={(15,0)}]
\foreach \row/\col in {0/7,1/7,2/4,3/4,4/1,5/1}{
  \fill[fill=violet!22] (\col,\row) rectangle +(1,1);}
\drawpartitiongrid{0/8,1/8,2/5,3/5,4/2,5/2}{8}
\node at (4,7.4) {$\bigl(\delta_{3,2}^{(2,3)}\bigr)'$};
\end{scope}
\definecolor{gline}{RGB}{126,95,188}
\draw[->,thick,gline] (8.0,3.0) -- (14.0,3.0);
\end{tikzpicture}
\end{center}
The violet boxes in the second conjugate diagram mark the added vertical strip.

The shape $\delta_{k+1}^{(w,s)}$ is
$\mathbb P_{s,w}^{k}$ in~\cite{GKM}. For every $k\ge1$,
\begin{align*}
\left\lvert\delta_k^{(w,s)}\right\rvert=\frac{ws\,k(k-1)}{2},
\qquad
\ell\bigl(\delta_k^{(w,s)}\bigr)=s(k-1).
\end{align*}
Conjugation interchanges the two step parameters:
\begin{align*}
\bigl(\delta_k^{(w,s)}\bigr)'=\delta_k^{(s,w)}.
\end{align*}
The endpoints satisfy
\begin{align*}
\delta_{k,0}^{(w,s)}=\delta_k^{(w,s)},
\qquad
\delta_{k,s}^{(w,s)}=\delta_{k+1}^{(w,s)}.
\end{align*}
The sequences $(\delta_{k,g}^{(w,s)})_{g=0}^{s}$ and
$((\delta_{k,g}^{(w,s)})')_{g=0}^{s}$ are called the horizontal
sequence and the conjugate vertical sequence, respectively.

\subsection{Macdonald polynomials and principal specializations}

For a partition $\lambda$, let $P_\lambda$ denote the monic Macdonald polynomial. In $n$ variables, write
\[
P_\lambda(x_1,\ldots,x_n;q,t).
\]
The parameters $q,t$ are independent indeterminates. For $n\in\mathbb Z_{>0}$, write
\[
\mathbf t_n=(1,t,\ldots,t^{n-1}).
\]
A positive integer $n$ satisfying $n\ge\ell(\lambda)$ is called an admissible number of variables for $\lambda$.

\begin{theorem}[\cite{Macdonald}]\label{thm:macdonald-principal}
For every partition $\lambda$ and $n\ge\ell(\lambda)$,
\begin{align*}
P_\lambda(\mathbf t_n;q,t)
=t^{n(\lambda)}
\prod_{(i,j)\in\lambda}
\frac{1-q^{a'_\lambda(i,j)}t^{n-l'_\lambda(i,j)}}
     {1-q^{a_\lambda(i,j)}t^{l_\lambda(i,j)+1}}.
\end{align*}
\end{theorem}

For indeterminates $a,y$ and $r\in\mathbb Z_{\ge0}$, the $y$-shifted factorial is
\[
(a;y)_r=\prod_{i=0}^{r-1}(1-ay^i),
\qquad
(a;y)_0=1.
\]

The principal specialization has the equivalent product form~\cite{Schlosser}:
\begin{equation}\label{eq:macdonald-pair}
P_\lambda(\mathbf t_n;q,t)
=t^{n(\lambda)}
\prod_{1\le i<j\le n}
\frac{(t^{j-i+1};q)_{\lambda_i-\lambda_j}}
     {(t^{j-i};q)_{\lambda_i-\lambda_j}}.
\end{equation}
For integers $a\ge b\ge0$, the Gaussian polynomial is
\[
\begin{bmatrix}a\\ b\end{bmatrix}_t
=\frac{(t;t)_a}{(t;t)_b(t;t)_{a-b}}.
\]

For a positive integer $r$, the basic hypergeometric series~\cite{GasperRahman} is
\[
{}_r\phi_{r-1}\!\left(\begin{matrix}a_1,\ldots,a_r\\
b_1,\ldots,b_{r-1}\end{matrix};t,z\right)
=\sum_{g\ge0}\frac{(a_1;t)_g\cdots(a_r;t)_g}
{(t;t)_g(b_1;t)_g\cdots(b_{r-1};t)_g}z^g.
\]
For $a\in\mathbb C$ and $g\ge0$, the ordinary shifted factorial is
$(a)_g=\prod_{i=0}^{g-1}(a+i)$. The series ${}_rF_{r-1}$ uses the
same convention with ordinary shifted factorials, replacing
$(t;t)_g$ by $g!$.

\subsection{Notation for special functions}

For $s\in\mathbb Z_{>0}$, the little $q$-Jacobi polynomial~\cite{KLS} is
\[
p_s(x;a,b\mid t)={}_2\phi_1\!\left(
\begin{matrix}t^{-s},ab t^{s+1}\\at\end{matrix};t,tx\right).
\]
For $\alpha>0$, the monic Jack polynomial~\cite{StanleyJack} satisfies
\begin{align*}
P_\lambda^{(\alpha)}(x_1,\ldots,x_n)
=\lim_{t\to1^-}P_\lambda(x_1,\ldots,x_n;t^\alpha,t).
\end{align*}

\section{Principal specialization ratios}

\subsection{Strip sequences and rectangular complementation}

The size and length of $\delta_{k,g}^{(w,s)}$ are
\[
\left\lvert\delta_{k,g}^{(w,s)}\right\rvert
=gwk+\frac{ws\,k(k-1)}{2},
\qquad
\ell\bigl(\delta_{k,g}^{(w,s)}\bigr)=s(k-1)+g.
\]
The conjugate parts satisfy
\begin{equation}\label{eq:column-lengths}
\bigl(\delta_{k,g}^{(w,s)}\bigr)'_{wr+j}=s(k-1-r)+g,
\qquad 0\le r\le k-1,\quad 1\le j\le w.
\end{equation}

For $0\le g<s$, the skew diagram $\delta_{k,g+1}^{(w,s)}/\delta_{k,g}^{(w,s)}$ consists of one box in each of the first $wk$ columns. The difference is a horizontal strip of size $wk$.

\begin{proposition}\label{prop:shortest}
Let $k,w,s\in\mathbb Z_{>0}$. The unique shortest sequence of partitions from $\delta_k^{(w,s)}$ to $\delta_{k+1}^{(w,s)}$ with horizontal strip differences has $s$ steps and is
\[
\bigl(\delta_{k,g}^{(w,s)}\bigr)_{g=0}^{s}.
\]
\end{proposition}

\begin{proof}
Each of the first $wk$ columns grows by $s$ boxes from $\delta_k^{(w,s)}$ to $\delta_{k+1}^{(w,s)}$. A horizontal strip adds at most one box to each column. At least $s$ steps are required. Equation~\eqref{eq:column-lengths} identifies every successive difference in $(\delta_{k,g}^{(w,s)})_{g=0}^{s}$ as a horizontal strip. This sequence attains the lower bound.

In any sequence of $s$ steps, each of the first $wk$ columns receives one box at every step. After $g$ steps, the length of column $wr+j$ is
\[
s(k-1-r)+g
\]
for $0\le r\le k-1$ and $1\le j\le w$. Every partition in the sequence is contained in $\delta_{k+1}^{(w,s)}$. By Equation~\eqref{eq:column-lengths}, the partition reached after $g$ steps is $\delta_{k,g}^{(w,s)}$.
\end{proof}

For $k=1$, $(\delta_{1,g}^{(w,s)})'=(g^w)$. For $k\ge2$,
Equation~\eqref{eq:column-lengths} determines
\begin{equation}\label{eq:conjugate-staircase}
\bigl(\delta_{k,g}^{(w,s)}\bigr)'
=\bigl((s(k-1)+g)^w,(s(k-2)+g)^w,\ldots,(s+g)^w,g^w\bigr),
\end{equation}
with zero parts omitted. The conjugate endpoints are
$\delta_k^{(s,w)}$ and $\delta_{k+1}^{(s,w)}$. Conjugation exchanges
horizontal and vertical strips, and Proposition~\ref{prop:shortest}
identifies the conjugate sequence as the unique shortest sequence
of vertical strips. For $0<g<s$, its interior partitions are not members of the
horizontal sequence with $w$ and $s$ interchanged: their largest parts
lie strictly between $s(k-1)$ and $sk$.

For $w,k\ge2$ and $0\le g<s$, the partition
$\delta_{k,g}^{(w,s)}$ agrees with
$\operatorname{St}(s,g;w;k-1)$ in~\cite[Section~6]{CDLfactor},
with trailing zeros omitted. For $1\le g\le s$, the conjugate
$(\delta_{k,g}^{(w,s)})'$ equals
$\operatorname{Qsc}(w,s,k;k,g)$ in~\cite[Section~6.1]{CGL},
again omitting trailing zeros.
These quasistaircase partitions index singular nonsymmetric Macdonald
polynomials at specialized parameters~\cite{CDsingular}.

For $k,w,s,n\in\mathbb Z_{>0}$, $0\le g\le s$, and
$n\ge s(k-1)+g$, the rectangular complement of
$\delta_{k,g}^{(w,s)}$ in $((wk)^n)$ is
$\delta_{k,n-s(k-1)-g}^{(w,s)}$ provided
$0\le n-s(k-1)-g\le s$. Indeed, reversing the $n$ parts and subtracting each from $wk$
recovers this partition. For $0\le h\le s$, this complement equals
$\delta_{k,h}^{(w,s)}$ exactly if
\[
n=s(k-1)+g+h.
\]

\subsection{Horizontal ratios and coefficientwise nonnegativity}

A ratio is independent of $q$ if its expansion at $q=0$ has no
terms with positive powers of $q$. For a partition $\mu$, set $\mu_1=0$ for $\mu=\varnothing$.
For $m\ge\max\{1,\mu_1\}$, write $(m^g,\mu)$ for the partition
obtained by adjoining $g$ rows of length $m$. For $0\le g<h$ and
$n\ge\ell(\mu)+h$, successive applications
of Theorem~\ref{thm:macdonald-principal} yield
\begin{equation}\label{eq:general-row-ratio}
\frac{P_{(m^h,\mu)}(\mathbf t_n;q,t)}{P_{(m^g,\mu)}(\mathbf t_n;q,t)}
=t^{(h-g)|\mu|+\frac{m(h-g)(g+h-1)}{2}}
\prod_{j=1}^m
\frac{(q^{j-1}t^{n-\mu'_j-h+1};t)_{h-g}}
{(q^{j-1}t^{\mu'_{m-j+1}+g+1};t)_{h-g}}.
\end{equation}
In one step, adjoining a row of length $m$ introduces the numerator
factor $1-q^{j-1}t^{n-\mu'_j-g}$ and the denominator factor
$1-q^{j-1}t^{\mu'_{m-j+1}+g+1}$ in column $j$.
The exponent of $t$ increases by $|\mu|+mg$. Multiplication from
$g$ to $h-1$ proves the formula.

For $\delta_{k,g}^{(w,s)}=((wk)^g,\delta_k^{(w,s)})$,
Equation~\eqref{eq:general-row-ratio} specializes to
\begin{equation}\label{eq:local-ratio}
\frac{P_{\delta_{k,g+1}^{(w,s)}}(\mathbf t_n;q,t)}
     {P_{\delta_{k,g}^{(w,s)}}(\mathbf t_n;q,t)}
=t^{gwk+\frac{ws\,k(k-1)}{2}}
\prod_{r=0}^{k-1}\prod_{j=1}^{w}
\frac{1-q^{j-1+wr}t^{n-s(k-1)-g+sr}}
     {1-q^{j-1+wr}t^{g+1+sr}}
\end{equation}
for $0\le g<s$ and $n\ge s(k-1)+g+1$. The ratio between stages
$0\le g<h\le s$ is
\begin{equation}\label{eq:arbitrary-gh}
\frac{P_{\delta_{k,h}^{(w,s)}}(\mathbf t_n;q,t)}
     {P_{\delta_{k,g}^{(w,s)}}(\mathbf t_n;q,t)}
=t^{\frac{wk(h-g)\{s(k-1)+g+h-1\}}{2}}
\prod_{r=0}^{k-1}\prod_{j=1}^{w}
\frac{(q^{j-1+wr}t^{n-s(k-1)-h+1+sr};t)_{h-g}}
     {(q^{j-1+wr}t^{g+1+sr};t)_{h-g}}
\end{equation}
for $n\ge s(k-1)+h$. These formulas follow by substituting the
column lengths from~\eqref{eq:column-lengths} into
\eqref{eq:general-row-ratio}.

\begin{theorem}\label{thm:general-nonnegativity}
Let $\lambda,\mu$ be partitions with at most $n$ parts. The series
\[
t^{n(\lambda)-n(\mu)}
\frac{P_\mu(\mathbf t_n;q,t)}{P_\lambda(\mathbf t_n;q,t)}
\]
belongs to $\mathbb Z[[q,t]]$. Coefficientwise nonnegativity
of this series forces independence of $q$.
\end{theorem}
\begin{proof}
Theorem~\ref{thm:macdonald-principal} expresses the series as a
finite product of factors $1-q^at^b$ and their inverses, with
$a\ge0$ and $b\ge1$. This proves the asserted integrality and
absolute convergence for $0<q,t<1$.
For each fixed $0\le q<1$, formula~\eqref{eq:macdonald-pair} gives
\[
\lim_{t\to1^-}t^{-n(\lambda)}P_\lambda(\mathbf t_n;q,t)
=\frac{n!}{\prod_{j\ge0}m_j(\lambda)!}.
\]
The normalized ratio has the same finite limit for
every $0\le q<1$. Assume coefficientwise nonnegativity. For fixed
$0<q<1$, a positive coefficient $c$ of $q^at^b$ with $a\ge1$
makes the difference between the series at $q$ and at $q=0$ at least
$cq^at^b$. As $t\to1^-$, this lower bound tends to $cq^a>0$,
contradicting equality of the limits. Every coefficient of positive
degree in $q$ vanishes.
\end{proof}

At $q=0$, formula~\eqref{eq:macdonald-pair} yields
\[
P_\lambda(\mathbf t_n;0,t)
=t^{n(\lambda)}\frac{(t;t)_n}{\prod_{j\ge0}(t;t)_{m_j(\lambda)}}.
\]
The factors from unequal parts telescope against those
from each maximal block of equal parts. Consequently,
\begin{equation}\label{eq:zero-multiplicity-ratio}
t^{n(\lambda)-n(\mu)}
\frac{P_\mu(\mathbf t_n;0,t)}{P_\lambda(\mathbf t_n;0,t)}
=\frac{\prod_{j\ge0}(t;t)_{m_j(\lambda)}}
       {\prod_{j\ge0}(t;t)_{m_j(\mu)}}.
\end{equation}

\begin{theorem}\label{thm:q-independence-criterion}
Let $\lambda,\mu$ be partitions with at most $n$ parts. The ratio
\[
\frac{P_\mu(\mathbf t_n;q,t)}{P_\lambda(\mathbf t_n;q,t)}
\]
is independent of $q$ precisely for pairs satisfying, for every
$d\ge1$ and $1\le c<n$,
\begin{equation}\label{eq:q-independence-criterion}
\#\{1\le i\le n-c:\lambda_i-\lambda_{i+c}>d\}
=\#\{1\le i\le n-c:\mu_i-\mu_{i+c}>d\}.
\end{equation}
For $\lambda_1-\lambda_n\ge2$,
condition~\eqref{eq:q-independence-criterion} forces equality of
the multisets
$\{\lambda_i-\lambda_{i+c}:1\le i\le n-c\}$ and
$\{\mu_i-\mu_{i+c}:1\le i\le n-c\}$ for every
$1\le c<n$. Under the same bound, the condition is equivalent to
\begin{equation}\label{eq:q-independent-monomial}
P_\mu(\mathbf t_n;q,t)
=t^{n(\mu)-n(\lambda)}P_\lambda(\mathbf t_n;q,t).
\end{equation}
\end{theorem}
\begin{proof}
By~\eqref{eq:macdonald-pair}, the quotient of the ratio by its
value at $q=0$ is a product of factors
\[
\frac{1-q^dt^{c+1}}{1-q^dt^c},\qquad d\ge1,\quad 1\le c<n.
\]
For each partition, this factor occurs once for every
$1\le i\le n-c$ with $\lambda_i-\lambda_{i+c}>d$ or,
respectively, $\mu_i-\mu_{i+c}>d$.
Condition~\eqref{eq:q-independence-criterion} cancels all factors.

For the converse, take the smallest $d\ge1$ violating
\eqref{eq:q-independence-criterion}. The factors of lower
degree in $q$ cancel. The coefficient of $q^d$ in the normalized
quotient is
\[
(1-t)\sum_{c=1}^{n-1}t^c\bigl(
\#\{1\le i\le n-c:\mu_i-\mu_{i+c}>d\}
-\#\{1\le i\le n-c:\lambda_i-\lambda_{i+c}>d\}
\bigr).
\]
The coefficient counts differ for at least one $c$, and this
polynomial is nonzero. The ratio depends on $q$.

Under $\lambda_1-\lambda_n\ge2$ and
\eqref{eq:q-independence-criterion}, the case $c=n-1$ determines
$\mu_1-\mu_n=\lambda_1-\lambda_n$. At
$d=\lambda_1-\lambda_n-1$, the counts in
\eqref{eq:q-independence-criterion} record the windows joining the
largest and smallest parts. Their distances are the consecutive integers from
$n-m_{\lambda_1}(\lambda)-m_{\lambda_n}(\lambda)+1$
to $n-1$. The total number of windows is
$m_{\lambda_1}(\lambda)m_{\lambda_n}(\lambda)$.
Their number of distinct distances and their total number determine
the unordered pair of multiplicities of the largest and smallest
parts. The corresponding pair for $\mu$ is the same.

For $1\le c<n$, counting the positive thresholds in each window yields
\[
\sum_{d\ge1}\#\{1\le i\le n-c:\lambda_i-\lambda_{i+c}>d\}=\sum_{i=1}^{n-c}(\lambda_i-\lambda_{i+c})-(n-c)+\sum_{j\ge0}\max\{m_j(\lambda)-c,0\}.
\]
The first sum on the right equals
$\sum_{i=1}^c(\lambda_i-\lambda_{n-c+i})$
and is unchanged under $c\leftrightarrow n-c$.
Applying equality of the left sides for $\lambda$ and $\mu$
at $c$ and $n-c$ yields
\begin{align*}
&\sum_{j\ge0}\bigl(\max\{m_j(\lambda)-c,0\}
                -\max\{m_j(\lambda)-(n-c),0\}\bigr)\\
&=\sum_{j\ge0}\bigl(\max\{m_j(\mu)-c,0\}
                -\max\{m_j(\mu)-(n-c),0\}\bigr).
\end{align*}
The equality extends to $c=0,n$ by the common number $n$ of parts. The second difference of $\max\{h-c,0\}$ at $c=j$
is $1$ for $h=j$ and $0$ otherwise. For
$\max\{h-n+c,0\}$, it is $1$ for $h=n-j$.
For $1\le j<n$, taking second differences yields
\begin{align*}
&\#\{h\ge0:m_h(\lambda)=j\}-\#\{h\ge0:m_h(\mu)=j\}\\
&=\#\{h\ge0:m_h(\lambda)=n-j\}
 -\#\{h\ge0:m_h(\mu)=n-j\}.
\end{align*}

For $j>\frac{n}{2}$, a block of size $j$ occurs at most once.
A block of this size in $\lambda$ but not in $\mu$ would force a
block of size $n-j$ in $\lambda$. These two blocks exhaust the
$n$ parts and must be the blocks of the largest and smallest parts.
Their unordered sizes agree with those for $\mu$. This forces a
block of size $j$ in $\mu$, a contradiction. Exchanging $\lambda$ and $\mu$
excludes the opposite discrepancy. The numbers of blocks of size
$j>\frac{n}{2}$ agree. The displayed symmetry forces equality
for $j<\frac{n}{2}$, and the total of $n$ parts determines
$j=\frac{n}{2}$ for even $n$.

The number of windows with zero difference at distance $c$ is
$\sum_{j\ge0}\max\{m_j(\lambda)-c,0\}$.
It is the same for both partitions. Condition
\eqref{eq:q-independence-criterion} determines the numbers of
windows with differences at least $2$. The zero counts agree, and
the remaining windows have difference $1$. The full multisets of
window differences coincide for each $c$. The multiplicities
agree as multisets, and~\eqref{eq:zero-multiplicity-ratio}
proves~\eqref{eq:q-independent-monomial}.
\end{proof}

For $\lambda_1-\lambda_n\le1$, condition~\eqref{eq:q-independence-criterion}
holds precisely for $\mu_1-\mu_n\le1$. By~\eqref{eq:zero-multiplicity-ratio}, the normalized ratio is
\[
\frac{\begin{bmatrix}n\\m_{\mu_1}(\mu)\end{bmatrix}_t}
{\begin{bmatrix}n\\m_{\lambda_1}(\lambda)\end{bmatrix}_t}.
\]
This rational function is regular at $t=1$, and all its possible poles are
roots of unity. If it is not a polynomial, a pole other than $1$ gives
radius of convergence $1$. Pringsheim's theorem excludes nonnegative
coefficients, since the function is regular at $t=1$. The converse
is conjectured in~\cite[Conjecture~1]{GatzweilerKrattenthaler}.
The cited paper proves the conjecture for Gaussian quotients with
lower indices differing by at most three or with the smaller lower
index at most two, after symmetry reduction.

\begin{theorem}\label{thm:nonnegative-complement}
Let $\mu$ be a partition, $m\ge\max\{2,\mu_1\}$,
$0\le g<h$, and $n\ge\ell(\mu)+h$.
The following conditions are equivalent.
\begin{enumerate}
\renewcommand{\labelenumi}{(\arabic{enumi})}
\item The ratio
$\displaystyle\frac{P_{(m^h,\mu)}(\mathbf t_n;q,t)}
{P_{(m^g,\mu)}(\mathbf t_n;q,t)}$
is coefficientwise nonnegative in $\mathbb Z[[q,t]]$.
\item The ratio is independent of $q$.
\item For every $1\le j\le m$,
\begin{equation}\label{eq:general-complement-columns}
\mu'_j+\mu'_{m-j+1}=n-g-h.
\end{equation}
\item The partition $(m^h,\mu)$ is the rectangular complement
of $(m^g,\mu)$ in $(m^n)$.
\end{enumerate}
Under these conditions, the ratio equals
\begin{equation}\label{eq:general-complement-power}
t^{\frac{m(h-g)(n-1)}{2}}.
\end{equation}
\end{theorem}
\begin{proof}
By Theorem~\ref{thm:general-nonnegativity}, (1) implies (2). The factor of degree zero in $q$ in
\eqref{eq:general-row-ratio} is nonzero in $\mathbb Q[[t]]$.
Independence of $q$ forces the product of factors involving
positive powers of $q$ to equal $1$.

The factor indexed by $j=1$ is independent of $q$.
After the factors with indices below $j$ have canceled, the
coefficient of $q^{j-1}$ in the remaining product is
\[
\left(t^{\mu'_{m-j+1}+g+1}-t^{n-\mu'_j-h+1}\right)
\sum_{v=0}^{h-g-1}t^v,
\qquad 2\le j\le m.
\]
All factors of larger index start at higher degrees in $q$.
The geometric sum is nonzero. Independence of $q$ forces
\eqref{eq:general-complement-columns} successively for
$j=2,\ldots,m$. The case $j=m$ establishes the reflected
identity for $j=1$.

The complement of $(m^g,\mu)$ in $(m^n)$ has column lengths
$n-g-\mu'_{m-j+1}$. Equality with the lengths
$h+\mu'_j$ of $(m^h,\mu)$ is exactly condition (3).
Under these equalities the factorials in
\eqref{eq:general-row-ratio} cancel. Summing the equalities
in~\eqref{eq:general-complement-columns} gives
$2|\mu|=m(n-g-h)$, proving~\eqref{eq:general-complement-power}.
It establishes (1) and (2) and completes the equivalence.
\end{proof}

\begin{corollary}\label{cor:row-complement}
Let $\mu$ be a partition with $\mu_1<m$, where $m\ge2$.
For $0\le g<h$ and
$n\ge\ell(\mu)+h$, the ratio
\[
\frac{P_{(m^h,\mu)}(\mathbf t_n;q,t)}
     {P_{(m^g,\mu)}(\mathbf t_n;q,t)}
\]
is coefficientwise nonnegative if and only if
\[
n=\ell(\mu)+g+h
\quad\text{and}\quad
\mu\text{ is its own complement in }(m^{\ell(\mu)}).
\]
Under these conditions, the ratio equals
$t^{\frac{m(h-g)(n-1)}{2}}$.
For $k,w,s\in\mathbb Z_{>0}$ with $wk\ge2$,
$0\le g<h\le s$, and $n\ge s(k-1)+h$, the generalized
staircase ratio is coefficientwise nonnegative precisely for
$n=s(k-1)+g+h$. Its value is
\[
\frac{P_{\delta_{k,h}^{(w,s)}}(\mathbf t_n;q,t)}
     {P_{\delta_{k,g}^{(w,s)}}(\mathbf t_n;q,t)}
=t^{\frac{wk(h-g)(n-1)}{2}}.
\]
\end{corollary}
\begin{proof}
Since $\mu'_m=0$ and $\mu'_1=\ell(\mu)$, the outer pair
in~\eqref{eq:general-complement-columns} gives
$n=\ell(\mu)+g+h$. The remaining pairs identify $\mu$ as its own complement
in $(m^{\ell(\mu)})$. The generalized staircase
$\delta_k^{(w,s)}$ equals its complement in
$((wk)^{s(k-1)})$, and
Theorem~\ref{thm:nonnegative-complement} gives the stated ratios.
\end{proof}

The inversion identity in~\cite[Corollary~3.3]{LuqueSubrectangular}
uses the $Q$ normalization. For $\lambda\subseteq(m^n)$, its monic form is
\[
P_{(m-\lambda_n,\ldots,m-\lambda_1)}(x_1,\ldots,x_n;q,t)
=(x_1\cdots x_n)^m P_\lambda(x_1^{-1},\ldots,x_n^{-1};q,t).
\]
To verify it, inversion of the variables transforms the first Macdonald
difference operator into $t^{n-1}$ times the operator with parameters
$q^{-1},t^{-1}$. Multiplication by $(x_1\cdots x_n)^m$ changes the
eigenvalue to
$\sum_{i=1}^n q^{-(m-\lambda_{n+1-i})}t^{-(n-i)}$.
Rectangular complementation preserves dominance order, and the
resulting polynomial is monic and triangular. The eigenvalues of the
first Macdonald difference operator are distinct over $\mathbb Q(q,t)$
on partitions with at most $n$ parts. The transformed polynomial
therefore equals
$P_{(m-\lambda_n,\ldots,m-\lambda_1)}(x_1,\ldots,x_n;q^{-1},t^{-1})$.
Monic Macdonald polynomials are invariant under simultaneous inversion
of $q$ and $t$~\cite{Macdonald}, completing the proof.

For $wk=1$, $k=w=1$ and $\delta_{1,g}^{(1,s)}=(1^g)$. For
$0\le g<h\le s$ and $n\ge h$, Formula~\eqref{eq:general-row-ratio} reduces to
\begin{equation}\label{eq:wk1-ratio}
\frac{P_{\delta_{1,h}^{(1,s)}}(\mathbf t_n;q,t)}
     {P_{\delta_{1,g}^{(1,s)}}(\mathbf t_n;q,t)}
=t^{\frac{(h-g)(g+h-1)}{2}}
\frac{(t^{n-h+1};t)_{h-g}}{(t^{g+1};t)_{h-g}}.
\end{equation}
At $g=0$, the partition is $(1^h)$, and $P_{(1^h)}=e_h$ has
specialization
\[
P_{(1^h)}(\mathbf t_n;q,t)
=\sum_{0\le i_1<\cdots<i_h<n}t^{i_1+\cdots+i_h},
\]
with nonnegative coefficients for every $n\ge h$.
This explains the restriction $wk\ge2$ in Corollary~\ref{cor:row-complement}.

At $q=0$, the monic Macdonald polynomial specializes to the
Hall--Littlewood polynomial, and only the factor with $(r,j)=(0,1)$
remains in the quotient of shifted factorials in
\eqref{eq:arbitrary-gh}. For
$k,w,s,n\in\mathbb Z_{>0}$, $0\le g\le s$, and
$n\ge s(k-1)+g$, the resulting Hall--Littlewood ratio is
\begin{align*}
\frac{P_{\delta_{k,g}^{(w,s)}}(\mathbf t_n;0,t)}
     {P_{\delta_k^{(w,s)}}(\mathbf t_n;0,t)}
=t^{\frac{wkg(g-1)}{2}+\frac{wgs\,k(k-1)}{2}}
\begin{bmatrix}n-s(k-1)\\ g\end{bmatrix}_t.
\end{align*}
The Gaussian polynomial enumerates partitions inside a rectangle~\cite{StanleyEC1}. The normalized ratio is
\begin{align*}
t^{-\frac{wkg(g-1)}{2}-\frac{wgs\,k(k-1)}{2}}
\frac{P_{\delta_{k,g}^{(w,s)}}(\mathbf t_n;0,t)}
     {P_{\delta_k^{(w,s)}}(\mathbf t_n;0,t)}
=\sum_{\nu\subseteq((n-s(k-1)-g)^g)}t^{\lvert\nu\rvert}.
\end{align*}

\subsection{Conjugate vertical ratios}

\begin{proposition}\label{prop:vertical-nonnegative}
For $m,n\in\mathbb Z_{>0}$ with $n\ge m$, let $\mu$
have at most $m$ parts. Write
$\mu+(h^m)=(\mu_1+h,\ldots,\mu_m+h)$, with zero parts
appended to $\mu$. For $0\le g<h$,
\begin{equation}\label{eq:general-column-ratio}
\frac{P_{\mu+(h^m)}(\mathbf t_n;q,t)}
     {P_{\mu+(g^m)}(\mathbf t_n;q,t)}
=t^{\frac{(h-g)m(m-1)}{2}}
\prod_{i=1}^m
\frac{(q^{\mu_i+g}t^{n-i+1};q)_{h-g}}
     {(q^{\mu_i+g}t^{m-i+1};q)_{h-g}}.
\end{equation}
Except for $\mu=\varnothing$, $g=0$, and $h=1$,
this ratio is coefficientwise nonnegative in $\mathbb Z[[q,t]]$
if and only if $n=m$. For $n=m$ the ratio equals
$t^{\frac{(h-g)m(m-1)}{2}}$.
In the exceptional case, the ratio equals
\[
t^{\frac{m(m-1)}{2}}\begin{bmatrix}n\\m\end{bmatrix}_t.
\]
This ratio is coefficientwise nonnegative for every $n\ge m$.
\end{proposition}
\begin{proof}
For $\lambda$ with at most $m$ parts, adding one column changes
$n(\lambda)$ by $\frac{m(m-1)}{2}$. In the pairwise formula
\eqref{eq:macdonald-pair}, factors indexed by $1\le i<j\le m$
are unchanged. The equality $\lambda_j=0$ for $j>m$ makes the factors with
$m<j\le n$ telescope, for each $i\le m$, to
\[
\prod_{j=m+1}^n
\frac{1-q^{\lambda_i}t^{j-i+1}}
     {1-q^{\lambda_i}t^{j-i}}
=\frac{1-q^{\lambda_i}t^{n-i+1}}
      {1-q^{\lambda_i}t^{m-i+1}}.
\]
Iterating this identity proves
\eqref{eq:general-column-ratio}.

At $n=m$ every numerator factor cancels its denominator.
For $n>m$, suppose a positive exponent occurs among
$\mu_i+g+v$ for $1\le i\le m$ and $0\le v<h-g$.
Let $r$ be its least positive value. After the factors independent of
$q$ have been removed, the coefficient of $q^r$ is
\[
(1-t^{n-m})
\sum_{\substack{1\le i\le m,\ 0\le v<h-g\\\mu_i+g+v=r}}
t^{m-i+1}.
\]
This coefficient is nonzero. Theorem~\ref{thm:general-nonnegativity}
excludes nonnegativity. All exponents vanish only in the stated
exceptional case. Its ratio is the elementary symmetric
specialization from Theorem~\ref{thm:macdonald-principal}.
\end{proof}

For the conjugate staircase sequence, put $m=wk$ and
$\mu=(\delta_{k,0}^{(w,s)})'$. Its successive ratios are
\begin{align*}
\frac{P_{(\delta_{k,g+1}^{(w,s)})'}(\mathbf t_n;q,t)}
     {P_{(\delta_{k,g}^{(w,s)})'}(\mathbf t_n;q,t)}
=t^{\frac{wk(wk-1)}{2}}
\prod_{r=0}^{k-1}\prod_{j=1}^{w}
\frac{1-q^{sr+g}t^{n-wk+wr+j}}
     {1-q^{sr+g}t^{wr+j}}
\end{align*}
for $0\le g<s$, $n\ge wk$. For $0\le g<h\le s$,
\begin{align*}
\frac{P_{(\delta_{k,h}^{(w,s)})'}(\mathbf t_n;q,t)}
     {P_{(\delta_{k,g}^{(w,s)})'}(\mathbf t_n;q,t)}
=t^{\frac{(h-g)wk(wk-1)}{2}}
\prod_{r=0}^{k-1}\prod_{j=1}^{w}
\frac{(q^{sr+g}t^{n-wk+wr+j};q)_{h-g}}
     {(q^{sr+g}t^{wr+j};q)_{h-g}}.
\end{align*}
The row lengths in~\eqref{eq:conjugate-staircase} determine the indexed factors. Proposition~\ref{prop:vertical-nonnegative}
classifies these ratios, with the sole exceptional case
$(k,g,h)=(1,0,1)$.

\section{Boundary generating polynomials}

The terminal generalized staircase has $sk$ nonzero parts. The boundary
specialization of its horizontal sequence uses $n=sk$, the smallest
admissible number of variables for all stages. For $wk\ge2$,
Theorem~\ref{thm:nonnegative-complement} gives coefficientwise
nonnegativity of the ratio between stages $g<h$ exactly for $g+h=s$.

\subsection{Hypergeometric representation and complementation}

Equation~\eqref{eq:general-row-ratio} expresses the normalized stage values
as products of shifted factorials. At $n=sk$, these values are the
coefficients of a terminating basic hypergeometric polynomial.

For $k,w,s\in\mathbb Z_{>0}$, the ratios in
\eqref{eq:arbitrary-gh} give the terminating basic hypergeometric
representation
\begin{equation}\label{eq:boundary-hypergeom}
\sum_{g=0}^s
\frac{P_{\delta_{k,g}^{(w,s)}}(\mathbf t_{sk};q,t)}
{P_{\delta_k^{(w,s)}}(\mathbf t_{sk};q,t)}z^g
={}_{wk}\phi_{wk-1}\!\left(
\begin{matrix}
\{q^{-(j-1+wr)}t^{-s(r+1)}\}_{r,j}\\
\{q^{j-1+wr}t^{1+sr}\}_{(r,j)\ne(0,1)}
\end{matrix};t,\;(-1)^{wk}q^{\binom{wk}{2}}t^{wsk^2}z\right).
\end{equation}
The indices in the parameter lists range over $0\le r<k$ and
$1\le j\le w$. The entry $t^{-s}$ terminates the series.
To obtain this formula, reverse the numerator factorials in
\eqref{eq:arbitrary-gh} using
\[
(q^{j-1+wr}t^{s(r+1)-g+1};t)_g
=(-1)^g q^{(j-1+wr)g}
 t^{s(r+1)g-\frac{g(g-1)}{2}}
 (q^{-(j-1+wr)}t^{-s(r+1)};t)_g.
\]
The factor indexed by $(r,j)=(0,1)$ contributes $(t;t)_g$
in the denominator, and the other powers yield the argument in
\eqref{eq:boundary-hypergeom}.

At $n=sk$ and $0<q,t<1$, Equation~\eqref{eq:local-ratio} becomes
\[
\frac{P_{\delta_{k,g+1}^{(w,s)}}(\mathbf t_{sk};q,t)}{P_{\delta_{k,g}^{(w,s)}}(\mathbf t_{sk};q,t)}=t^{\frac{wsk(k-1)}{2}}\prod_{r=0}^{k-1}\prod_{j=1}^{w}\frac{t^g-q^{j-1+wr}t^{s(r+1)}}{1-q^{j-1+wr}t^{1+sr+g}}.
\]
As $g$ increases, $t^g$ decreases. Each numerator decreases
and each positive denominator increases. The successive ratios
decrease strictly with $g$.

At $n=sk$, the complement of $\delta_{k,g}^{(w,s)}$
in $((wk)^{sk})$ is $\delta_{k,s-g}^{(w,s)}$.
The corresponding principal specialization identity is
\begin{equation}\label{eq:boundary-complement}
\frac{P_{\delta_{k,s-g}^{(w,s)}}(\mathbf t_{sk};q,t)}
{P_{\delta_k^{(w,s)}}(\mathbf t_{sk};q,t)}
=t^{\frac{wk(sk-1)(s-2g)}{2}}
\frac{P_{\delta_{k,g}^{(w,s)}}(\mathbf t_{sk};q,t)}
{P_{\delta_k^{(w,s)}}(\mathbf t_{sk};q,t)},
\qquad 0\le g\le s.
\end{equation}
For $wk\ge2$, Theorem~\ref{thm:nonnegative-complement} establishes
Equation~\eqref{eq:boundary-complement} for the stages $g$ and $s-g$.
For $g>s-g$, the ratio is inverted. In the column case
$wk=1$, the same identity is the symmetry of the Gaussian polynomial
in Equation~\eqref{eq:wk1-ratio}. The centered generating polynomial
\begin{equation}\label{eq:boundary-centered}
\sum_{g=0}^s t^{-\frac{wk(sk-1)g}{2}}
\frac{P_{\delta_{k,g}^{(w,s)}}(\mathbf t_{sk};q,t)}
{P_{\delta_k^{(w,s)}}(\mathbf t_{sk};q,t)}z^g
\end{equation}
is reciprocal, with constant and leading coefficients equal to $1$.

\subsection{Negative real zeros}

\begin{proposition}[\cite{KSV}]\label{prop:classical-multiplier}
Let $s\in\mathbb Z_{>0}$ and
\[
A(z)=\sum_{i=0}^s\binom{s}{i}a_i z^i,
\qquad
B(z)=\sum_{i=0}^s\binom{s}{i}b_i z^i
\]
be real polynomials of degree $s$ with only negative real
zeros. Their Schur--Szeg\H{o} composition
\[
\sum_{i=0}^s\binom{s}{i}a_i b_i z^i
\]
has only negative real zeros. Multiplying the coefficient of $z^i$ by $b_i$ preserves real
zeros for polynomials of degree at most $s$.
For $s\ge2$, the smallest ratio of consecutive absolute zeros of the
composition is at least the larger of the corresponding smallest ratios
for $A$ and $B$.
\end{proposition}

\begin{proposition}[\cite{ASW}]\label{prop:classical-pf}
Let $\sum_{i=0}^s a_i z^i$ be a polynomial with nonnegative real
coefficients. Its zeros are all real if and only if every minor of the
Toeplitz matrix $(a_{j-i})_{i,j\ge0}$ is nonnegative, with
$a_i=0$ for $i<0$ or $i>s$.
\end{proposition}

\begin{lemma}\label{lem:local-multiplier}
Let $s\ge1$ and $0<a,b,t<1$. The polynomial
\begin{equation}\label{eq:local-multiplier}
M_s(z;a,b,t)=\sum_{g=0}^s\binom{s}{g}
 t^{\frac{g(g-s)}{2}}
\frac{(b;t)_s}{(a;t)_g(b;t)_{s-g}}z^g
\end{equation}
has $s$ simple negative real zeros. For $s\ge2$, their
absolute values $0<x_1<\cdots<x_s$ satisfy
\begin{align*}
\frac{x_{i+1}}{x_i}>\frac{1}{t},\qquad 1\le i<s.
\end{align*}
The factors multiplying $\binom{s}{g}$ in
\eqref{eq:local-multiplier} form a multiplier sequence for
polynomials of degree at most $s$.
\end{lemma}
\begin{proof}
Set $M_0=1$. For $s\ge1$,
\begin{equation}\label{eq:local-recurrence}
(1-a)M_s(z;a,b,t)=(1-bt^{\frac{s-1}{2}}z)M_{s-1}(t^{-\frac{1}{2}}z;at,bt,t)+(t^{\frac{1-s}{2}}z-a)M_{s-1}(t^{\frac{1}{2}}z;at,bt,t).
\end{equation}
The coefficient of $z^g$ on the right side is
\[
(1-a)t^{\frac{g(g-s)}{2}}
\frac{(b;t)_s}{(a;t)_g(b;t)_{s-g}}
\left(\binom{s-1}{g}+\binom{s-1}{g-1}\right).
\]
Binomial coefficients with indices outside $0,\ldots,s-1$ are zero. Pascal's identity
establishes~\eqref{eq:local-recurrence}, including the two endpoint
coefficients.

For $s=1$, the polynomial is $1+\frac{1-b}{1-a}z$.
The induction hypothesis applies to $M_{s-1}(z;at,bt,t)$. At $z=-t^{\frac{1}{2}}x_i$, the first term
of~\eqref{eq:local-recurrence} vanishes. Its second term has sign
$(-1)^i$. At $z=-t^{-\frac{1}{2}}x_i$, the second term
vanishes and the first has the same sign $(-1)^i$.
The signs follow from the location of $-tx_i$ and
$-\frac{x_i}{t}$ relative to the zeros of $M_{s-1}$.
There is one zero in each of the $s$ disjoint intervals
\[
(-t^{\frac{1}{2}}x_1,0),\qquad
(-t^{\frac{1}{2}}x_{i+1},-t^{-\frac{1}{2}}x_i)
\quad(1\le i<s-1),\qquad
(-\infty,-t^{-\frac{1}{2}}x_{s-1}).
\]
The extreme intervals use the positive constant and leading
coefficients. These $s$ intervals exhaust the degree, proving simplicity.
The absolute values of successive zeros lie on opposite sides of
$t^{\frac{1}{2}}x_i$ and $t^{-\frac{1}{2}}x_i$.
Their ratio exceeds $\frac{1}{t}$, proving the strict separation. The multiplier criterion
in Proposition~\ref{prop:classical-multiplier} gives the final claim.
\end{proof}

\begin{theorem}\label{thm:boundary-real}
Let $\mu$ be a partition, $m,s\in\mathbb Z_{>0}$ with
$m>\mu_1$, and put $n=\ell(\mu)+s$. For $0<q,t<1$, the polynomial
\begin{equation}\label{eq:general-boundary}
\sum_{g=0}^s
\frac{P_{(m^g,\mu)}(\mathbf t_n;q,t)}
     {P_\mu(\mathbf t_n;q,t)}z^g
\end{equation}
has $s$ simple negative real zeros. For $s\ge2$, the
absolute values $0<x_1<\cdots<x_s$ of consecutive zeros satisfy
\begin{equation}\label{eq:boundary-mesh}
\frac{x_{i+1}}{x_i}\ge\frac{1}{t},\qquad 1\le i<s.
\end{equation}
For $m\ge2$ the inequalities are strict. For $m=1$ they
are equalities. The coefficients form a finite P\'olya frequency
sequence.
\end{theorem}
\begin{proof}
Center the $g$th coefficient of~\eqref{eq:general-boundary}
by multiplying it by $t^{-g\{|\mu|+\frac{m(s-1)}{2}\}}$.
Equation~\eqref{eq:general-row-ratio} gives
\begin{equation}\label{eq:general-boundary-factorization}
t^{-g\{|\mu|+\frac{m(s-1)}{2}\}}\frac{P_{(m^g,\mu)}(\mathbf t_n;q,t)}{P_\mu(\mathbf t_n;q,t)}=t^{\frac{mg(g-s)}{2}}\prod_{j=1}^m\frac{(q^{j-1}t^{\ell(\mu)-\mu'_j+1};t)_s}{(q^{j-1}t^{\mu'_{m-j+1}+1};t)_g(q^{j-1}t^{\ell(\mu)-\mu'_j+1};t)_{s-g}}.
\end{equation}
All parameters in these factors lie in $(0,1)$.
For $j=1$, both parameters of the shifted factorials equal $t$. The corresponding
polynomial is
\begin{align*}
\sum_{g=0}^s t^{\frac{g(g-s)}{2}}
\begin{bmatrix}s\\g\end{bmatrix}_t z^g
=\prod_{v=0}^{s-1}(1+t^{v-\frac{s-1}{2}}z).
\end{align*}
The Gaussian recurrence establishes the product identity. Its zeros
are simple and negative, with consecutive absolute ratios
$\frac{1}{t}$. The centered polynomial with coefficients in
\eqref{eq:general-boundary-factorization} is the Schur--Szeg\H{o}
composition of this product and, for $2\le j\le m$, the polynomials
\[
M_s\bigl(z;q^{j-1}t^{\mu'_{m-j+1}+1},
q^{j-1}t^{\ell(\mu)-\mu'_j+1},t\bigr).
\]
The binomial normalization in the composition removes the factor
$\binom{s}{g}$ from each polynomial. The remaining coefficient
factors give~\eqref{eq:general-boundary-factorization}.
By Lemma~\ref{lem:local-multiplier}, every factor indexed by
$2\le j\le m$ has simple negative zeros, with a minimum ratio
strictly greater than $1/t$ for $s\ge2$. Proposition~\ref{prop:classical-multiplier}
preserves negative real zeros under composition and bounds the
minimum ratio from below by the larger minimum ratio of the
factors. The factor indexed by $j=1$ has minimum ratio $1/t$.
For $m\ge2$, the composition therefore has minimum ratio
strictly greater than $1/t$. For $m=1$, the polynomial equals the displayed product and has
consecutive absolute zero ratios $1/t$.
Positive rescaling of $z$ preserves these ratios.
Proposition~\ref{prop:classical-pf} gives the P\'olya frequency
property.
\end{proof}

Corollary~3.12 of~\cite{MPM} gives real zeros and, for $m\ge2$,
simplicity through strict parameter interlacing for the polynomial
in~\eqref{eq:general-boundary}.
Reversing the numerator factorials in~\eqref{eq:general-boundary-factorization}
produces the upper parameters $t^{-s}$ and
$q^{1-j}t^{\mu'_j-\ell(\mu)-s}>t^{1-s}$ for $2\le j\le m$.
The lower parameters $q^{j-1}t^{\mu'_{m-j+1}+1}$ lie in $(0,1)$.
These satisfy the hypotheses of the cited result, after a change
of sign and a positive rescaling of the argument.
Lemma~\ref{lem:local-multiplier} and
Proposition~\ref{prop:classical-multiplier} establish the separation
bound~\eqref{eq:boundary-mesh}.
For $\mu=\varnothing$, $m=s=2$, and $n=3$, the polynomial
in~\eqref{eq:general-boundary} has nonreal zeros at
$(q,t)=(\frac{1}{2},\frac{9}{10})$. The same holds for
$\mu=(2)$, $m=s=2$, and $n=3$, with $m=\mu_1$. The two examples show the necessity of
$n=\ell(\mu)+s$ and $m>\mu_1$, respectively.

The specialization $m=wk$ and $\mu=\delta_k^{(w,s)}$ in
\eqref{eq:general-boundary-factorization} has parameters
$q^{j-1+wr}t^{1+sr}$.

\begin{proposition}\label{prop:staircase-boundary-bounds}
For $k,w,s\in\mathbb Z_{>0}$, $0\le g\le s$, and $0<q,t<1$,
\[
\begin{bmatrix}s\\g\end{bmatrix}_t
\le t^{-\frac{wkg\{s(k-1)+g-1\}}{2}}
\frac{P_{\delta_{k,g}^{(w,s)}}(\mathbf t_{sk};q,t)}
     {P_{\delta_k^{(w,s)}}(\mathbf t_{sk};q,t)}
\le\left(\prod_{r=0}^{k-1}
\frac{\begin{bmatrix}s(r+1)\\g\end{bmatrix}_t}
     {\begin{bmatrix}sr+g\\g\end{bmatrix}_t}\right)^w.
\]
The lower and upper bounds are the limits as $q\to0^+$ and
$q\to1^-$, respectively. Both inequalities are strict for
$wk\ge2$ and $0<g<s$. Equality holds throughout for $wk=1$
or $g\in\{0,s\}$.
\end{proposition}
\begin{proof}
For $0<g<s$, the factors
\[
\frac{1-q^{j-1+wr}t^{1+sr+s-g+v}}
     {1-q^{j-1+wr}t^{1+sr+v}},\qquad 0\le v<g,
\]
in~\eqref{eq:general-boundary-factorization} increase strictly
with $q$ for $j-1+wr>0$.
At $q=0$, only the factor indexed by $(r,j)=(0,1)$ remains,
giving the lower bound. The limit $q\to1^-$ gives the upper
bound. For $wk=1$ no $q$-dependent factors occur.
At $g=0$ all three expressions equal $1$. At $g=s$ the
middle expression equals $1$ by~\eqref{eq:boundary-complement},
and both bounds equal $1$.
\end{proof}

\subsection{Reciprocity and gamma coefficients}

\begin{proposition}\label{prop:boundary-reciprocity}
Let $\mu$ be a partition, $m,s\in\mathbb Z_{>0}$ with
$m>\mu_1$, and put $n=\ell(\mu)+s$.
The centered polynomial
\[
\sum_{g=0}^s
 t^{-g\{|\mu|+\frac{m(s-1)}{2}\}}
\frac{P_{(m^g,\mu)}(\mathbf t_n;q,t)}
     {P_\mu(\mathbf t_n;q,t)}z^g
\]
is reciprocal as an identity in the independent parameters $q,t$
if and only if $\mu$ is its own complement in
$(m^{\ell(\mu)})$.
\end{proposition}
\begin{proof}
Complementation gives
$\ell(\mu)-\mu'_j=\mu'_{m-j+1}$, making every factor in
\eqref{eq:general-boundary-factorization} invariant under
$g\mapsto s-g$. Conversely, the leading coefficient of a
reciprocal polynomial with constant term $1$ must be $1$. Hence
\[
\frac{P_{(m^s,\mu)}(\mathbf t_n;q,t)}
     {P_\mu(\mathbf t_n;q,t)}
=t^{s\{|\mu|+\frac{m(s-1)}{2}\}}.
\]
For $m\ge2$, Theorem~\ref{thm:nonnegative-complement}
forces complementation. For $m=1$, the condition $m>\mu_1$
forces $\mu=\varnothing$.
\end{proof}

For a generalized staircase, equality of $\delta_k^{(w,s)}$
with its complement in $((wk)^{s(k-1)})$ gives
\eqref{eq:boundary-centered}.

The $\gamma$-basis of reciprocal polynomials~\cite{Athanasiadis}
and the zero separation in Theorem~\ref{thm:boundary-real}
give coefficient bounds for generalized staircases.

\begin{corollary}
Let $k,w,s\in\mathbb Z_{>0}$, $0<q,t<1$, and write $s=2r$ or
$s=2r+1$ with $r\in\mathbb Z_{\ge0}$. The centered boundary polynomial
has the unique expansion
\begin{align*}
\sum_{g=0}^s t^{-\frac{wk(sk-1)g}{2}}
\frac{P_{\delta_{k,g}^{(w,s)}}(\mathbf t_{sk};q,t)}
{P_{\delta_k^{(w,s)}}(\mathbf t_{sk};q,t)}z^g
=\sum_{i=0}^{r}\gamma_i z^i(1+z)^{s-2i}.
\end{align*}
For $0\le i\le r$, the coefficients satisfy $\gamma_0=1$
and the following bounds. For $s=2r$,
\[
\gamma_i\ge [u^i]\prod_{j=1}^{r}
\bigl(1+(t^{-j+\frac{1}{2}}+t^{j-\frac{1}{2}}-2)u\bigr).
\]
For $s=2r+1$,
\[
\gamma_i\ge [u^i]\prod_{j=1}^{r}
\bigl(1+(t^{-j}+t^j-2)u\bigr).
\]
For $1\le i\le r$, these bounds are positive.
All equalities hold for $wk=1$. For $wk\ge2$ and $s\ge2$, the
inequalities are strict for $1\le i\le r$.
The polynomial $\sum_{i=0}^r\gamma_i u^i$ has $r$ distinct negative
zeros. Its coefficients form a P\'olya frequency sequence. For $r\ge2$,
the absolute zeros have consecutive ratios strictly greater than $\frac{1}{t}$.
\end{corollary}
\begin{proof}
Write the absolute values of the negative zeros of
\eqref{eq:boundary-centered} as $0<x_1<\cdots<x_s$.
Reciprocity implies $x_i x_{s+1-i}=1$.
For odd $s$, the middle zero is $-1$. Pairing reciprocal factors gives
\[
(1+z)^{s-2r}\prod_{j=1}^r
\left((1+z)^2+(x_{s-r+j}+x_{s-r+j}^{-1}-2)z\right).
\]
The numbers $x_{s-r+j}+x_{s-r+j}^{-1}-2$ are positive and strictly
increasing. The coefficients $\gamma_i$ are the elementary symmetric functions
of these numbers. Their generating polynomial
$\sum_i\gamma_i u^i$ has distinct negative zeros.
For $s=2r$, reciprocal pairing and \eqref{eq:boundary-mesh} imply
$x_{r+j}\ge t^{-j+\frac{1}{2}}$. For $s=2r+1$, they imply
$x_{r+1+j}\ge t^{-j}$. The function $x+x^{-1}-2$ is strictly
increasing on $x>1$. Comparison of its values in the paired factors
proves the two bounds.
Theorem~\ref{thm:boundary-real} identifies the equality case and
strictness.

By Proposition~\ref{prop:classical-pf}, the sequence $(\gamma_i)_{i=0}^r$ is a P\'olya frequency sequence. For $1<x<y$,
\[
\frac{y+y^{-1}-2}{x+x^{-1}-2}
=\frac{y}{x}\left(\frac{1-y^{-1}}{1-x^{-1}}\right)^2>\frac{y}{x}.
\]
The zeros of $\sum_{i=0}^{r}\gamma_i u^i$ are the reciprocals of these positive factors. Ratios of consecutive absolute zeros exceed those of the corresponding
boundary zeros and are greater than $\frac{1}{t}$.
\end{proof}

\subsection{Specializations and Jack limits}

For $wk=2$, the cases $(w,k)=(1,2)$ and $(2,1)$ have the same form
of successive boundary ratios:
\[
\frac{P_{\delta_{k,g+1}^{(w,s)}}(\mathbf t_{sk};q,t)}
{P_{\delta_{k,g}^{(w,s)}}(\mathbf t_{sk};q,t)}
=t^{2g+sk-s}\frac{(1-t^{s-g})(1-qt^{sk-g})}
{(1-t^{g+1})(1-qt^{sk-s+g+1})}.
\]
The series in~\eqref{eq:boundary-hypergeom} reduces to the little
$q$-Jacobi polynomial of Section~2:
\begin{align*}
\sum_{g=0}^s
\frac{P_{\delta_{k,g}^{(w,s)}}(\mathbf t_{sk};q,t)}
{P_{\delta_k^{(w,s)}}(\mathbf t_{sk};q,t)}z^g
=p_s\!\left(qt^{2sk-1}z;\,qt^{sk-s},q^{-2}t^{-2sk-1}\mid t\right).
\end{align*}

For the $q$-Hahn weights~\cite[Section~14.6]{KLS} with parameter
$q^{-1}t^{-sk-1}$, the quotient of consecutive normalized weights
is $qt$ times the Macdonald ratio above. Both sequences equal $1$ at $g=0$. Their coefficients satisfy
\begin{align*}
(qt)^g
\frac{P_{\delta_{k,g}^{(w,s)}}(\mathbf t_{sk};q,t)}
{P_{\delta_k^{(w,s)}}(\mathbf t_{sk};q,t)}
=\begin{bmatrix}s\\g\end{bmatrix}_t
\frac{(q^{-1}t^{-sk};t)_g(q^{-1}t^{-sk};t)_{s-g}}{(q^{-1}t^{-sk})^g(q^{-1}t^{-sk};t)_s}.
\end{align*}

The parameter $q^{-1}t^{-sk}$ lies outside the usual positive
orthogonality range, although the displayed weights are positive for
$0<q,t<1$.

The limit $q=t^\alpha$, $t\to1^-$ preserves the coefficient
symmetry in~\eqref{eq:boundary-complement}. The zeros in the Jack
limit follow from the Jacobi identity below.

\begin{lemma}[\cite{KLS}]\label{lem:jacobi-zeros}
For $s\ge1$ and $a,b>-1$, the Jacobi polynomial
$P_s^{(a,b)}$ has $s$ simple zeros in $(-1,1)$.
For $a,b>0$, the Jacobi polynomial expansion yields
\begin{align*}
\sum_{g=0}^s\binom{s}{g}
\frac{(b)_s}{(a)_g(b)_{s-g}}z^g
=\frac{s!}{(a)_s}(1-z)^s
P_s^{(a-1,b-1)}\!\left(\frac{1+z}{1-z}\right).
\end{align*}
The polynomial on either side has $s$ simple negative zeros.
\end{lemma}

\begin{proposition}
Let $\mu$ be a partition, $m,s\in\mathbb Z_{>0}$ with
$m>\mu_1$, $n=\ell(\mu)+s$, and $\alpha>0$.
The Jack boundary polynomial
\begin{align*}
\sum_{g=0}^s
\frac{P_{(m^g,\mu)}^{(\alpha)}(1,\ldots,1)}
     {P_\mu^{(\alpha)}(1,\ldots,1)}z^g
\end{align*}
has $s$ simple negative zeros for $m\ge2$. For $m=1$
it is $(1+z)^s$. Its coefficients are
\begin{equation}\label{eq:jack-general-coefficients}
\prod_{j=1}^m
\frac{(1+\ell(\mu)-\mu'_j+\alpha(j-1))_s}
{(1+\mu'_{m-j+1}+\alpha(j-1))_g
 (1+\ell(\mu)-\mu'_j+\alpha(j-1))_{s-g}}.
\end{equation}
\end{proposition}
\begin{proof}
Substitute $q=t^\alpha$ into~\eqref{eq:general-boundary-factorization}
and take $t\to1^-$. The powers of $1-t$ cancel,
establishing~\eqref{eq:jack-general-coefficients}.
The $j=1$ coefficient factor in
\eqref{eq:jack-general-coefficients} is $\binom{s}{g}$ and
generates $(1+z)^s$. For $2\le j\le m$, take
\[
a=1+\mu'_{m-j+1}+\alpha(j-1),\qquad
b=1+\ell(\mu)-\mu'_j+\alpha(j-1).
\]
By Lemma~\ref{lem:jacobi-zeros}, the polynomial
\[
\sum_{g=0}^s\binom{s}{g}\frac{(b)_s}{(a)_g(b)_{s-g}}z^g
\]
has simple negative zeros. By Proposition~\ref{prop:classical-multiplier},
multiplication of the coefficient of $z^g$ by
$(b)_s/((a)_g(b)_{s-g})$ preserves real-rootedness in degrees
at most $s$. For $m\ge2$, the first such multiplier acts on
$(1+z)^s$, and the remaining ones preserve strict separation. For $s=1$, simplicity follows directly from
positivity of the two coefficients. For $m=1$, no other factor occurs.
\end{proof}

For $m=wk$, $\mu=\delta_k^{(w,s)}$, and $0\le g<s$, the successive Jack ratio is
\begin{equation}\label{eq:jack-ratio}
\frac{P_{\delta_{k,g+1}^{(w,s)}}^{(\alpha)}(1,\ldots,1)}
{P_{\delta_{k,g}^{(w,s)}}^{(\alpha)}(1,\ldots,1)}
=\prod_{r=0}^{k-1}\prod_{j=1}^{w}
\frac{s(r+1)-g+\alpha(j-1+wr)}
{g+1+sr+\alpha(j-1+wr)}.
\end{equation}
The corresponding generating polynomial is a terminating
${}_{wk}F_{wk-1}$ series with upper parameters
$-s(r+1)-\alpha(j-1+wr)$, lower parameters
$1+sr+\alpha(j-1+wr)$ outside $(r,j)=(0,1)$,
and argument $(-1)^{wk}z$. Its coefficients are symmetric,
by~\eqref{eq:boundary-complement} and the Jack limit.
Ratios from the hook length formula for ordinary staircases and
specializations involving Jacobi polynomials are treated in~\cite{ShimazakiJacobi}.

For $wk=2$, the Jack boundary coefficients are normalized symmetric
Hahn weights~\cite[Section~9.5]{KLS}. The case $wk=2$ of~\eqref{eq:jack-ratio} agrees with the quotient
of consecutive Hahn weights at parameter $-\alpha-sk-1$. Both
normalized sequences equal $1$ at $g=0$. The identity is
\begin{align*}
\frac{P_{\delta_{k,g}^{(w,s)}}^{(\alpha)}(1,\ldots,1)}
{P_{\delta_k^{(w,s)}}^{(\alpha)}(1,\ldots,1)}
=\binom{s}{g}\frac{(-\alpha-sk)_g(-\alpha-sk)_{s-g}}{(-\alpha-sk)_s},
\qquad 0\le g\le s.
\end{align*}
The expression is symmetric under $g\mapsto s-g$.

For fixed $0<q<1$, as $t\to1^-$, the factors in
\eqref{eq:local-ratio} involving positive powers of $q$ tend to $1$.
The factor indexed by $(r,j)=(0,1)$ tends to
$\frac{s-g}{g+1}$. The boundary coefficients converge to
$\binom{s}{g}$, and the generating polynomial tends to $(1+z)^s$.

\section*{Acknowledgments}
This work was supported by JSPS KAKENHI Grant Number JP26K24505.

\end{document}